\documentclass[a4,twoside,12pt]{amsart}

\usepackage{a4wide}
\usepackage{amsfonts}
\usepackage{amsmath}
\usepackage{amssymb}
\usepackage{amsthm}
\usepackage[english]{babel}
\usepackage{enumerate}
\usepackage{fontenc}
\usepackage[latin1]{inputenc}
\usepackage[english]{varioref}
\usepackage{url}

\title[Intersection cohomology and Severi's varieties]
{Intersection cohomology and Severi's varieties}

\author{Vincenzo Di Gennaro }
\address{Universit\`a di Roma \lq\lq Tor Vergata\rq\rq, Dipartimento di Matematica,
Via della Ricerca Scientifica, 00133 Roma, Italy.}
\email{digennar@axp.mat.uniroma2.it}

\author{Davide Franco }
\address{Universit\`a di Napoli
\lq\lq Federico II\rq\rq, Dipartimento di Matematica e
Applicazioni \lq\lq R. Caccioppoli\rq\rq, Via Cintia, 80126
Napoli, Italy.} \email{davide.franco@unina.it}

\theoremstyle{plain}
\newtheorem{theorem}{Theorem}[section]
\newtheorem{corollary}[theorem]{Corollary}

\newtheorem{proposition}[theorem]{Proposition}
\newtheorem{definition}[theorem]{Definition}

\theoremstyle{definition}

\newtheorem{remark}[theorem]{Remark}
\newtheorem{notations}[theorem]{Notations}

\DeclareMathOperator{\con}{Con}
\DeclareMathOperator{\sing}{Sing}

\newcommand{\vc}{\ensuremath{\mathcal{V}}}
\newcommand{\ic}{\ensuremath{\mathcal{I}}}
\newcommand{\oc}{\ensuremath{\mathcal{O}}}

\newcommand{\fc}{\ensuremath{\mathcal{F}}}

\newcommand{\hc}{\ensuremath{\mathcal{H}}}
\newcommand{\cc}{\ensuremath{\mathcal{C}}}

\newcommand{\xc}{\ensuremath{\mathcal{X}}}

\newcommand{\bP}{\mathbb{P}}
\newcommand{\bPd}{{\mathbb{P}^{\vee}}}
\newcommand{\vX}{X^{\vee}}
\newcommand{\bQ}{\mathbb{Q}}
\newcommand{\bC}{\mathbb{C}}

\begin{document}

\begin{abstract} Let $X^{2n}\subseteq \bP ^N$ be a smooth projective
variety. Consider the intersection cohomology complex of the local
system $R^{2n-1}\pi{_*}\mathbb{Q}$, where $\pi$ denotes the
projection from the universal hyperplane family of $X^{2n}$ to
${(\bP ^N)}^{\vee}$. We investigate the cohomology of the
intersection cohomology complex $IC(R^{2n-1}\pi{_*}\mathbb{Q})$ over
the points of a Severi's variety, parametrizing nodal hypersurfaces,
whose nodes impose independent conditions on the very ample linear
system giving the embedding in $\bP ^N$.


\end{abstract}

\maketitle

\medskip
\begin{center}
{\it{Dedicated to Ciro Ciliberto on his seventieth birthday.}}
\end{center}
\medskip

\section{Introduction}

In the last years a great deal of work has been devoted to focusing
on the  deep relationship among  Hodge conjecture and singularities
of normal functions (compare e.g. with \cite{GG}, \cite{KP} and
references therein). The theory  of normal functions, which dates
back to Poincar\'e, Lefschetz and Hodge, had a renewed interest in
last years  after a crucial remark of Green and Griffiths \cite{GG}
that the Hodge conjecture is equivalent to the existence of
appropriately defined singularities for the normal function defined
by means of a primitive Hodge cycle.

One of the starting points of this remark is a fundamental result of
Kleiman concerning the smoothability of algebraic cycles of
intermediate dimension (\cite{Kleiman}, \cite[Example
15.3.2]{Fulton}). In light of this, R. Thomas showed inductively
that the Hodge conjecture reduces to the following statement
concerning middle dimensional Hodge cycles: {\it for all even
dimensional smooth complex projective varieties $(X^{2n},\oc(1))$
and any class $A\in H^{n,n}(X, \bC)\cap H^{2n}(X, \bQ)$, there is a
nodal hypersurface $D\subset X$ in $\mid \oc_X(N)\mid$ for some $N$,
such that the Poincar\'e dual of $A$ is in the image of the
pushforward map $H_{2n}(D, \bQ) \to  H_{2n}(X, \bQ)$} \cite{Thomas}.

In their fundamental work \cite{GG}, Green and Griffiths further
clarified the question by relating it  to the singularities of
normal functions. Specifically, they showed that the above
hypersurface $D$ would be a singular point of the normal function
associated to the Hodge class $A$, thus reducing the Hodge
conjecture to the existence of such singularities.

In the paper \cite{GG}, one can find various definitions of
singularities of normal functions. Some of them are formalized by de
Cataldo and Migliorini by means of the Decomposition Theorem
\cite{DeCMHodgeConj}. More precisely, let $X=X^{2n}$ be a
$2n$-dimensional smooth complex projective variety embedded in
$\mathbb{P}^N_{\mathbb{C}}\equiv\mathbb{P}$ via a complete linear
system $\mid H \mid$. Denote by $X^{\vee}$ the \emph{dual variety}
of $X$ and consider the \textit{universal hyperplane family}
$$ X \stackrel{q}{\longleftarrow} \mathcal{X} \stackrel{\pi}{\longrightarrow} \bPd, \quad \quad \dim \xc = 2n-1+N.
$$
The hyperplane sections  are the fibers of $\pi$:
$$\xc _H:= \pi^{-1}(H), \quad \quad \dim \xc_H=2n-1.$$

We observe that the projection from the  universal hyperplane family
$\mathcal{X} \stackrel{\pi}{\longrightarrow} \bPd$ is a proper map
with equidimensional fibres and it is smooth outside the dual
variety $X^{\vee}$.

As explained in \cite[Sec. 2]{DeCMHodgeConj}, the Decomposition Theorem for $\pi$ provides a non-canonical decomposition
\begin{equation}
R\pi{_*}\mathbb{Q}_{\xc}\cong \bigoplus_{i\in \mathbb{Z}}\bigoplus
_{j\in \mathbb{N}}IC(L_{ij})[-i-(2n-1+N)], \quad \text{in} \,\,\,
D_c^b(\bPd),
\end{equation}
where $L_{ij}$ denotes a local system on a suitable stratum of
codimension $j$ in $\bPd$. By \cite{DeCMHodgeConj}, the most
important summand for our purposes is the one related to
$L_{00}=(R^{2n-1}\pi{_*}\mathbb{Q}_{\xc})_{\mid U}$ ($U:= \bPd
\backslash X^{\vee}$) \cite[2.5]{DeCMHodgeConj}. Specifically, if
one fix  an intermediate primitive Hodge class $\xi$, de Cataldo and
Migliorini define the \textit{singular locus} of the normal function
of $\xi$ as the support of the component of $\xi$ belonging to
$\hc^{-N+1} IC(R^{2n-1}\pi{_*}\mathbb{Q}_{\xc})$ in the
decomposition above (which is well defined in view of the perverse
filtration) \cite[Definition 3.3, Remark 3.4]{DeCMHodgeConj}. The
main result of \cite{DeCMHodgeConj} says such a singularity is able
to detect the local triviality of $\xi$.

In view of the aforementioned work of Thomas, it is particularly
interesting to take a closer look at the cohomology sheaf
$\hc^{-N+1} IC(R^{2n-1}\pi{_*}\mathbb{Q}_{\xc})$, especially in
correspondence of nodal hypersurfaces. This paper, which is a small
step in this direction, is devoted to the study of $
IC(R^{2n-1}\pi{_*}\mathbb{Q}_{\xc})_D$ with $D$ \textit{nodal
divisor whose nodes impose independent conditions on the linear
system giving the embedding $X\subseteq \bP$}. Under this
hypothesis, what we are going to do is \textit{to compute the
cohomology of the complex $IC(R^{2n-1}\pi{_*}\mathbb{Q}_{\xc})_D$
and, above all, to give it a geometric interpretation}.

The hypothesis that the nodes impose independent conditions to the
hypersurfaces of a very ample linear system has long been
investigated in relation to the study of Severi's varieties,
parametrizing irreducible nodal curves on a smooth algebraic surface
(compare e.g. with \cite{Wahl}, \cite{Sernesi}, \cite{CS} and
references therein). In particular, in the paper \cite{Sernesi},
Sernesi proves that such a condition implies that Severi's variety
is smooth of the expected dimension, by using deformation theory.

Our first result consists in a different proof of the same result,
extended to smooth algebraic varieties of even dimension (compare
with Theorem \ref{first} and Remark \ref{rmkfirst}). Our approach,
which is independent of deformation theory  and consists in a
careful local study of the conormal map, allows us to prove that
\textit{the dual variety, in a neighborhood of a nodal hypersurface
whose nodes impose independent conditions, is a divisor with normal
crossings} (cf. Theorem \ref{first}).

Understanding the local structure of the dual manifold allows us in
Section 3 to compute the cohomology of the complex
$IC(R^{2n-1}\pi{_*}\mathbb{Q}_{\xc})_D$, where $D$ denotes a  nodal
hypersurface whose nodes pose independent conditions, and above all
to give it a geometric interpretation. More precisely, we see that
the cohomology is concentrated in degrees $-N$ and $-N+1$,  that
$\hc^{-N} IC(R^{2n-1}\pi{_*}\mathbb{Q}_{\xc})_D$ is naturally
isomorphic to $H^{2n-1}(D)$, and that $\hc^{-N+1}
IC(R^{2n-1}\pi{_*}\mathbb{Q}_{\xc})_D$ is related with the
\textit{defect} of the nodes (Remark \ref{defect}). Furthermore, we
prove that \textit{the perverse filtration of $H^{2n}(D)$ is as
simple as possible because it consists of only two pieces that vary
in local systems over any component of Severi's variety} (Corollary
\ref{cormain}).

As a by-product, we see that \textit{under the hypothesis that a
hyperplane cuts $X$ in a set of nodes imposing independent
conditions, the pull-back of a primitive Hodge cycle coincides with
its local Green-Griffiths invariant}. In particular, we get  a
different proof of the fact that \textit{the local Green-Griffiths
invariant detects the local triviality of a Hodge cycle}, in our
context \cite[Proposition 3.8 (ii)]{DeCMHodgeConj}.

Last but not least, in  Section 4 we provide several examples of
even-dimensional smooth projective varieties equipped with linear
systems containing nodal hypersurfaces $D$, whose nodes impose
independent conditions and such that $\hc^{-N+1}
IC(R^{2n-1}\pi{_*}\mathbb{Q}_{\xc})_D$ is non-trivial.

\bigskip

\section{Notations and basic facts}

\medskip

\begin{notations}\label{notations}
From now on, unless  otherwise stated, all cohomology and
intersection cohomology groups are with $\mathbb Q$-coefficients.
\begin{enumerate}
\item For a complex  algebraic variety $X$, we
denote by $H^{l}(X)$ and $IH^{l}(X)$ its cohomology and intersection
cohomology groups. Let $D_c^b(X)$ be the constructible derived
category of sheaves of $\mathbb Q$-vector spaces on $X$. For a
complex of sheaves $\mathcal F^{}\in D_c^b(X)$, we denote by
$\hc^l(\fc)$ the $l$-th cohomology sheaf of $\fc$ and by $\mathbb
H^l(\mathcal F^{})$ the $l$-th hypercohomology group of $\fc$. Let
$IC^{}_{X}$ denotes the intersection cohomology complex of $X$. If
$X$ is nonsingular, we have $IC^{}_{X}\cong \mathbb
Q_X[\dim_{\mathbb C} X]$, where $\mathbb Q_X$ is the constant sheaf
$\mathbb Q$ on $X$.
\item   More generally, let $i:S\to X$ be a locally closed embedding of a
smooth irreducible subspace of $X$ and let $L$ be a local system on
$S$. Denote by $IC_{\overline{S}}(L):=i_{!*}L[\dim S]\in D_c^b(X)$
the \textit{intersection cohomology complex} of $L$ \cite[Sec.
5.2]{Dimca}, \cite[Sec. 2.7]{DeCMReview}. It  is defined as the
\textit{intermediary extension} of $L$, that is the unique extension
of $L$ in $D_c^b(X)$ with neither subobjects nor quotients supported
on $\overline{S} \backslash S$.
\end{enumerate}
\end{notations}
\medskip

In this paper, $X$ denotes a $2n$-dimensional smooth complex
projective variety embedded in
$\mathbb{P}^N_{\mathbb{C}}\equiv\mathbb{P}$ via a complete linear
system $\mid H \mid$. Denote by $X^{\vee}$ the \emph{dual variety}
of $X$ and consider the \textit{universal hyperplane family}
$\mathcal{X}\subset X\times \bPd$. We have natural projections:
$$ X \stackrel{q}{\longleftarrow} \mathcal{X} \stackrel{\pi}{\longrightarrow} \bPd, \quad \quad \dim \xc = 2n-1+N.
$$
The hyperplane sections  are the fibers of $\pi$:
$$\xc _H:= \pi^{-1}(H)=X\cap H, \quad \quad \dim \xc_H=2n-1.$$
Let $\con(X)\subset \mathcal{X}$ be the \emph{conormal variety of $X$}:
$$\con(X):=\{(p,H)\in X\times \bPd: \,\, TX_p\subseteq H \}, \quad \quad  \dim \con(X) = N-1,$$
where $TX_p$ denotes the embedded tangent space to $X$ at $p$. We
denote by $\pi_1: \con(X) \to \bPd$ the restriction of
$\pi:\mathcal{X} \stackrel{}{\longrightarrow} \bPd$ to $\con(X)$. We
have $X^{\vee}=\pi_1 (\con(X))$.  If $\dim X^{\vee}< N-1$, then, for
a general $H\in X^{\vee}$, the fiber $\pi_1(H)^{-1}$ has positive
dimension. This means that the general tangent hyperplane to $X$ is
tangent to a subvariety of positive dimension. Obviously, this
cannot happen if we replace $\mid H \mid$ with a sufficiently large
multiple. So we can always assume that $X^{\vee}$ \textit{is a
hypersurface} of $\bPd$.

If $H \in U:=\bPd \backslash X^{\vee}$, then $\xc_H$ is smooth hence
$\pi: \pi^{-1}(U) \rightarrow U$ is a smooth fibration. This implies
that the sheaf $R^{2n-1}\pi{_*}\mathbb{Q}_{\xc}$ restricts to a
local system on $U$.

As explained in \cite[Sec. 2]{DeCMHodgeConj}, the Decomposition Theorem for $\pi$ provides a non-canonical decomposition
\begin{equation}
\label{decThm}  R\pi{_*}\mathbb{Q}_{\xc}\cong \bigoplus_{i\in \mathbb{Z}}\bigoplus _{j\in \mathbb{N}}IC(L_{ij})[-i-(2n-1+N)], \quad \text{in} \,\,\, D_c^b(\bPd),
\end{equation}
where $L_{ij}$ denote a local system on a suitable stratum of
codimension $j$ in $\bPd$. By \cite{DeCMHodgeConj}, the most
important summand for the purpose of detect the primitive Hodge
classes of $X$, is the one related to
$L_{00}=(R^{2n-1}\pi{_*}\mathbb{Q}_{\xc})_{\mid U}$
\cite[2.5]{DeCMHodgeConj}.

The main aim of this paper is to investigate the cohomology sheaves
of the complex $IC((R^{2n-1}\pi_* \bQ_{\xc })_{ \mid U})[-N]\in
D_c^b(\bPd)$, near the points corresponding to \textit{nodal
divisors}.

One can find different approaches to the Decomposition Theorem
\cite{BBD}, \cite{DeCMHodge}, \cite{DeCMReview}, \cite{Saito},
\cite{Williamson}, which is a very general result but also rather
implicit. On the other hand,  there are many special cases for which
the Decomposition Theorem admits a simplified and explicit approach.
One of these is  the case of varieties with isolated singularities
\cite{N}, \cite{DGF3}, \cite{DGF2strata}, \cite{Forum}, \cite{DGF4}.
For instance, in the work  \cite{DGF3}, one can find a simplified
approach to the Decomposition Theorem for varieties with isolated
singularities, in connection with the existence of a \textit{natural
Gysin morphism}, as defined in \cite[Definition 2.3]{DGF1}. One of
the main ingredients of these arguments is a generalization of the
Leray-Hirsch   theorem in a categorical framework (\cite[Theorem
7.33]{Voisin}, \cite[Lemma 2.5]{DGF2}).
\bigskip

\section{Local study of the dual hypersurface}

\medskip

\noindent Let $X$ be a $2n$-dimensional smooth complex projective
variety embedded in $\mathbb{P}^N_{\mathbb{C}}\equiv\mathbb{P}$ via
a complete linear system $\mid H \mid$. Assume moreover that
$X^{\vee}$, the \emph{dual variety} of $X$, is a hypersurface in
$\mathbb{P}^{\vee}$ (compare with the previous section).

With notations as above, for any $H\in X^{\vee}$, the corresponding hyperplane cuts $X$ in a singular divisor
$\xc _H:= \pi^{-1}(H)=X\cap H$.

\medskip

\begin{definition}
Let $\vX _r\subseteq \vX$ be the \emph{locally closed} Zariski
subset of $\vX$ containing all the hyperplanes $H$ s.t. $\sing X_H$
consists in exactly $r$ singular ordinary double  points ($r$
\emph{nodes} for short). An irreducible component of  $\vX _r$ is
said to have the \emph{the expected dimension} if its dimension is
equal to $N-r$ (in what follows we assume $r \leq N$). If $H\in \vX
_r$, then we say that $\vX _r$ \emph{has the expected dimension at}
$H$, if it belongs to a component having the expected dimension and
not in components having bigger dimension.
\end{definition}

\medskip
The following result is probably well known via deformation theory, we include a slightly different proof of it in the attempt
of making the present paper reasonably self-contained.

\begin{theorem}
\label{first} Assume that $\Delta:= \sing(\xc _H)$ consists of
$\delta$ nodes. If $\Delta$ imposes independent conditions to $\mid
H \mid$, that is if $H^1(\ic _{\Delta, X}(1))=0$, then for every
small ball $B\subseteq \bPd$ containing $H$ we have:

\begin{enumerate}
\item $B\cap \vX$ is a divisor of $B$ with normal crossings;
\item for every $r\leq\delta$, $B\cap \vX _r$ is non empty smooth of pure dimension $N-r$.
\end{enumerate}
In particular, $\vX _{\delta}$ is smooth and has the expected dimension at $H$.
\end{theorem}
\begin{proof}

By~\cite[Proposition 3.3]{N:K}, the projection
$\pi_1:\con(X)\to~\bPd$ is unramified at $(x_i,H)$, where the
$x_i$'s are the nodes of $X_H$ (observe that $(x_i, H)\in \con(X)$,
$\forall 1\leq i \leq \delta $). Hence, $\pi_1$ provides an
embedding of a suitable analytic neighborhood $U_i$ of $(x_i, H)$ in
$\bPd$: $$(x_i, H)\in U_i \subset \con(X).$$

Hence, the image $\mathcal{U} _i:=\pi_1 (U_i)$ provides a  branch of
$\vX$ passing through $H$, $\forall i$ such that $1\leq i \leq
\delta $.

By ~\cite[Lemme 4.1.2]{N:K}, each $\mathcal{U} _i$  intersects
$X^{\vee}\backslash {\text{Sing}}(X^{\vee})$ and we can find  a
sequence $(p_n^i,H_n^i)\in U_i$ such that
$$H_n^i\in X^{\vee}\backslash {\text{Sing}}(X^{\vee})\,\,\,\,\, \text{and} \,\,\,\,\, (p_n^i,H_n^i)\rightarrow (x_i,H).$$
By \cite[p. 209]{Harris}, the embedded tangent  space of
$\mathcal{U} _i$ at $H_n^i$  is  $p_n^i$ viewed as a hyperplane of
$\bPd$. Hence, the embedded tangent space of the branch $\mathcal{U}
_i$ at $H$ (in the following denoted by $T_{H,\, \mathcal{U} _i}$)
is $x_i$ (viewed as a hyperplane of $\bPd$).

On the other hand, the hypothesis $H^1(\ic _{\Delta, X}(1))=0$, applied to the short exact sequence
$$0 \rightarrow \ic _{\Delta, X}(1)\rightarrow \oc _X(1) \rightarrow \oc_{\Delta}(1) \rightarrow 0,$$
implies that $h^0(\ic _{\Delta, X}(1))=N +1 - \delta$ and the nodes span a linear subspace of dimension $\delta -1$ in $\bP$.
By our previous argument, we have
$$ \bigcap _{i=1}^{\delta} T_{H,\, \mathcal{U} _i}= < x_1, \, \dots  \, , x_{\delta} >^{\vee}=\mathbb{P}(H^0(\ic _{\Delta, X}(1)))
\quad \text{and} \quad \dim\bigcap _{i=1}^{\delta} T_{H,\,
\mathcal{U} _i}=N-\delta.$$ In other words, the branches of $\vX$ at
$H$ are independent and $\vX$ is a divisor with normal crossings
around $H$.

Consider now a subset $I\subseteq \{ 1, \dots , \delta \}$ and set
$$\vX _I:=\bigcap_{i\in I}\, \mathcal{U} _i.$$
As the branches $\mathcal{U} _i$ have independent tangent
hyperplanes at $H$, then for every sufficiently small ball
$B\subseteq \bPd$ containing $H$, $B \cap \vX _I$ is a smooth
complete intersection. Furthermore,  we have
\begin{equation}\label{2ways}
\dim B\cap \vX _I = N- \mid I\mid\,\,\,\,\,\,\, \text{and}
\,\,\,\,\,\,\, \dim B\cap \mathcal{U} _j \cap \vX _I= N-1- \mid
I\mid, \,\,\,\, \forall j\not \in I.
\end{equation}
We point out that one could also deduce (\ref{2ways}) from the
factoriality of  the ring of holomorphic functions defined in $B$
\cite[p. 10]{GriHa}. Indeed, by  factoriality, each branch
$\mathcal{U}_i$ is defined in $B$ by an analytic function $f_i$.
Since $\vX$ is a divisor with normal crossings around $H$, the
sequence $f_1, \dots , f_{\delta}$ \textit{is a regular sequence of
analytic functions} in $B$. Thus, any subset of $f_1, \dots ,
f_{\delta}$ is regular as well and (\ref{2ways}) follows at once.

Finally, the following locally closed subset of $\vX$
$$  B\cap \vX _I \backslash \left(\bigcup_{j\not \in I}B\cap \mathcal{U} _j \cap \vX _I \right)$$
is a non-empty analytic subspace of $\vX _{\mid I\mid}$, with the expected dimension.
 \end{proof}

\medskip

\begin{remark}\label{rmkfirst}
The final part of the proof of Theorem \ref{first} shows that, under
the hypothesis $H^1(\ic _{\Delta, X}(1))=0$, \emph{the nodes of $H$
can be independently smoothed}. This fact is usually proved by means
of deformation theory (compare with \cite{Wahl}, \cite{Sernesi},
\cite{CS}). We preferred to use a different approach here, because
it is more suited to our purposes.

\end{remark}

\bigskip

\section{Intersection cohomology complex on Severi's varieties}
\medskip

\noindent
\begin{notations}\label{not2}
\begin{enumerate}
\item Assume now that the hypotheses of Theorem \ref{first} are verified. In particular,
there exists a tubular neighborhood $T$ of some connected component
$\cc$ of $\vX _{\delta}$ such that $T \cap \vX$ is a \textit{divisor
with normal crossings} in $T$. Set $T^{0}:=T\backslash (T \cap
\vX)$. The local system $(R^{2n-1}\pi_* \bQ)_{\mid T^0}$ has a
canonical extension to a holomorphic vector bundle $\vc$ on $T$ (see
\cite{Deligne} and \cite{Schnell}).
\item  Fix $H\in U$. Combining Hard Lefschetz Theorem with  Lefschetz Hyperplane Theorem we have
$$H^{2n}(\xc_H)\cong H^{2n-2}(\xc_H)\cong H^{2n-2}(X).$$
Hence, $R^{2n}\pi_* \bQ_{\xc}$ is a \textit{constant  system} on
$U$. Put $h:= h^{2n-2}(X)$.
\end{enumerate}
\end{notations}

\begin{remark}
\label{isotopy} By Thom's first isotopy lemma \cite[Theorem
5.2]{GWPL}, the family $\xc \mid_{\cc}$ is locally trivial thus both
$(R^{2n-1}\pi_* \bQ_{\xc})_{\mid\cc}$ and $(R^{2n}\pi_*
\bQ_{\xc})_{\mid\cc}$ are local systems over $\cc$. Nevertheless, in
the following Theorem we give a direct proof of this result, that we
believe of independent interest.
\end{remark}

\medskip
\begin{theorem}\label{main} Assume the hypotheses of Theorem \ref{first} are verified for $H\in \vX_{\delta}$.
Fix  a connected component $\cc$ of $\vX_{\delta}$ containing $H$.
With notations as above, we have an isomorphism of local systems on $\cc$
\begin{equation}\label{th1}
\hc^0 (IC((R^{2n-1}\pi_* \bQ_{\xc })_{ \mid U})[-N]) _{\mid\cc}\cong
(R^{2n-1}\pi_* \bQ_{\xc})_{\mid\cc},
\end{equation}
and a short exact sequence of local systems on $\cc$
\begin{equation}\label{th2}
0\to \hc^1 (IC((R^{2n-1}\pi_* \bQ_{\xc })_{ \mid U})[-N])
_{\mid\cc}\to (R^{2n}\pi_* \bQ_{\xc})_{\mid\cc}\to \mathbb{Q}^{h}
\to 0.
\end{equation}
Furthermore, we have
\begin{equation}\label{th3}
\hc^i (IC((R^{2n-1}\pi_* \bQ_{\xc })_{ \mid U})[-N]) _{\mid\cc}=0,
\quad \forall i\geq 2.
\end{equation}
In particular, for any nodal divisor $H\in \vX_{\delta}$, we have
\begin{itemize}
\item $h^0 (IC((R^{2n-1}\pi_* \bQ_{\xc })_{ \mid U})[-N]) _H=h^{2n-1}(\xc _H),$
\item $h^1 (IC((R^{2n-1}\pi_* \bQ_{\xc })_{ \mid U})[-N]) _H=h^{2n}(\xc _H)-h,$
\item $h^i (IC((R^{2n-1}\pi_* \bQ_{\xc })_{ \mid U})[-N]) _H=0,$ $\forall i\geq 2$.
\end{itemize}

\end{theorem}
\begin{proof}
Assume  that the hypotheses of Theorem \ref{first} are verified. In
particular, there exists a tubular neighborhood $T$ of $\cc$ such
that $T \cap \vX$ is a divisor with normal crossings in $T$. Set
$T^{0}:=T\backslash (T \cap \vX)$. The local system $(R^{2n-1}\pi_*
\bQ)_{\mid T^0}$ has a canonical extension to a local system $\vc$
on $T$ (see \cite{Deligne} and \cite{Schnell}).

Further, in a suitable neighborhood of $H\in \cc$, the equation of
$\vX$ has the form $t_1 \dots t_{\delta}=0$ and the local system
$(R^{2n-1}\pi_* \bQ_{\xc })_{ \mid T^0}$ has monodromy operators
$T_1, \dots , T_{\delta}$, with $T_i$ given by moving around the
hyperplane $t_i=0$. If we denote by $N_i$ the \textit{logarithm of
the monodromy operator} $T_i$, by \cite{CKS} and \cite[p. 322]{KP}
the cohomology  $$\hc^i (IC((R^{2n-1}\pi_* \bQ_{\xc })_{ \mid
U})[-N]) _{H}$$ of the intersection cohomology complex at $H\in \cc$
can be computed as the $i$-th cohomology of the complex of
finite-dimensional vector spaces
$$B^p:=\bigoplus _{i_1<i_2< \dots < i_p}N_{i_1}N_{i_2}\dots N_{i_p}\vc_H, $$
with differential acting on the summands by the rule
$$N_{i_1}\dots \hat{N}_{i_l}\dots N_{i_{p+1}}\vc_H \stackrel{(-1)^{l-1} N_{i_l}}{\longrightarrow} N_{i_1}\dots N_{i_l}\dots N_{i_{p+1}}\vc_H.$$
Fix $H\in \vX_{\delta}$. Since $\xc _H$ is nodal, the logarithm of
the monodromy operators $N_i$ act according to the
\textit{Picard-Lefschetz formula}. Furthermore, as $\xc _H$ has
$\delta$ ordinary double points, the vanishing spheres are disjoint
to each other and we have
$$N_i N_j=0, \quad \text{for any} \quad i\not= j.$$
So the complex above is concentrated in degrees $0$ and $1$. We have
$$\hc^i (IC((R^{2n-1}\pi_* \bQ_{\xc })_{ \mid U})[-N]) _{H}=0,   \quad \forall H\in \vX_{\delta}, \,\,\, \forall i\geq 2.$$
Thus  (\ref{th3}) is  proved.

Furthermore, we have the following exact sequence
\begin{equation}\label{computeIC}
0\to \hc^0 (IC((R^{2n-1}\pi_* \bQ_{\xc })_{ \mid U})[-N]) _{H}\to
\vc_H \to  \bigoplus_i N_i\vc _H \to
\quad\quad\quad\quad\quad\quad\quad\quad\quad\quad
\end{equation}
$$
\quad\quad\quad\quad\quad\quad\quad\quad\quad\quad\quad\quad\quad\to
\hc^1 (IC((R^{2n-1}\pi_* \bQ_{\xc })_{ \mid U})[-N]) _{H} \to 0.
$$
On the other hand, consider a hyperplane $H_t\in U$ very near to
$H$, such that $\xc_{H_t}$ is smooth, and denote by $B_i$ a small
ball around the $i$-th node of $\xc_{H}$.  By excision, we have
$$H^l(\xc_{H}, \cup _i (\xc_{H}\cap B_i))\cong  H^l(\xc_{H_t}, \cup _i (\xc_{H_t}\cap B_i)).$$
From the exact sequence
$$ \dots \rightarrow H^{l-1}( \cup _i (\xc_{H}\cap B_i)) \rightarrow H^l(\xc_{H}, \cup _i (\xc_{H}\cap B_i))\rightarrow
H^l(\xc_{H})\rightarrow H^{l}( \cup _i (\xc_{H}\cap B_i))\rightarrow \dots ,$$
and recalling the conic nature of isolated singularities of a divisor \cite{Milnor}, we get
$$H^l(\xc_{H})\cong H^l(\xc_{H}, \cup _i (\xc_{H}\cap B_i))\cong H^l(\xc_{H_t}, \cup _i (\xc_{H_t}\cap B_i)), \,\, \text{if} \,\, l\geq 2.$$
Inserting in the relative cohomology sequence for the pair
$(\xc_{H_t}, \cup _i (\xc_{H_t}\cap B_i))$, we find the exact
sequence
\begin{equation}
\label{computeIC2}  0 \to   H^{2n-1}(\xc_{H})\to
H^{2n-1}(\xc_{H_t})\to  H^{2n-1}( \cup _i (\xc_{H_t}\cap
B_i))\cong\mathbb{Q}^{\delta}\to\quad\quad\quad\quad\quad\quad\quad\quad\quad
\end{equation}
$$\quad\quad\quad\quad\quad\quad\quad\quad\quad\quad\quad\quad\quad\quad\quad\quad\quad\quad\to  H^{2n}(\xc_{H})
\to   H^{2n-2}(X)\to 0,$$ where we have also taken into account  the
isomorphism (compare with Notations \ref{not2})
$$ H^{2n-2}(X)\cong  H^{2n}(\xc _{H_t}).$$
By the Picard-Lefschetz formula, the map
$H^{2n-1}(\xc_{H_t})\to\mathbb{Q}^{\delta}$ coincides, up to some
irrelevant sign, with the map $\vc_H \to  \bigoplus_i N_i\vc _H $ of
the sequence (\ref{computeIC}). Hence, comparing (\ref{computeIC})
and (\ref{computeIC2}), we find
$$\hc^0 (IC((R^{2n-1}\pi_* \bQ_{\xc })_{ \mid U})[-N]) _{H}\cong H^{2n-1}(\xc_{H})\cong (R^{2n-1}\pi_* \bQ_{\xc})_{H} $$
and
$$0\to \hc^1 (IC((R^{2n-1}\pi_* \bQ_{\xc })_{ \mid U})[-N]) _{H}\to (R^{2n}\pi_* \bQ_{\xc})_{H}\to \mathbb{Q}^{h}
\to 0.
$$
Thus, (\ref{th1}) and (\ref{th2}) are proved for any $H\in \cc$ and
we need only to prove a similar result at the level of local systems
on $\cc$.
First of all, the fact that $T\cap \vX$ is a divisor with normal
crossings implies that, if we fix a small neighborhood $B$ of $H$ in
$\bPd$, then the local fundamental group $\pi_1(B\backslash(B\cap
\vX), H)\cong \mathbb{Z}^{\delta}$ is independent of $H\in
\vX_{\delta}$. So, the rank of $\vc_{\mid \vX_{\delta}} \to
\bigoplus_i N_i\vc _{\mid\vX_{\delta}}$ does not change as long as
we let $H$ to vary in $\cc$, and its kernel is a vector bundle on
$\cc$. Furthermore, from the description of the canonical extension
given e.g. in \cite[sec. 2]{Schnell}, one infers that such a kernel
is the $\mathbb{Z}^{\delta}$-invariant part  of the local system
$(R^{2n-1}\pi_* \bQ_{\xc })_{ \mid T^0}\otimes \mathbb{C}$.
Moreover, since the tubular neighborhood $T$ is homeomorphic to a
fiber bundle on $\vX_{\delta}$, the long exact sequence of homotopy
groups of $T^0$
$$\dots \rightarrow \mathbb{Z}^{\delta}\rightarrow \pi_1(T^0, H) \rightarrow \pi_1(\cc, H)\rightarrow 0$$
shows that the kernel of $\vc_{\mid \cc} \to \bigoplus_i N_i\vc
_{\mid\cc}$ descends to a local system on $\cc$, consisting in the
$\mathbb{Z}^{\delta}$-invariant part  of the local system
$(R^{2n-1}\pi_* \bQ_{\xc })_{ \mid T^0}\otimes \mathbb{C}$, i.e.
with $(R^{2n-1}\pi_* \bQ_{\xc})_{\mid\vX_{\delta}}\otimes
\mathbb{C}$ by taking also into account of (\ref{computeIC2}). This
concludes the proof of (\ref{th1}).

As for the proof of (\ref{th2}), we argue in a similar way. First of
all we observe that the vector space $\bigoplus_i N_i\vc _{H}$  is
contained in the $\mathbb{Z}^{\delta}$-invariant part  of the local
system $(R^{2n-1}\pi_* \bQ_{\xc })_{ \mid T^0}\otimes \mathbb{C}$ as
well. This follows just combining \cite[p. 42]{Looijenga} (recall
that the dimension of $\xc_{H_t}$ is odd), with the fact that the
vanishing spheres are disjoint to each other. By the same argument
as above, $\bigoplus_i N_i\vc _{H}$ is the stalk at $H$ of a local
system $\cc$, on which $\pi_1(\cc, H)$ acts by  ``exchanging the
branches". By (\ref{computeIC}), $\hc^1 (IC((R^{2n-1}\pi_* \bQ_{\xc
})_{ \mid U})[-N]) _{H}$ is a local system $\cc$ as well and
(\ref{th2}) follows by comparison with (\ref{computeIC2}).
\end{proof}

\medskip
\begin{remark}
Theorem \ref{main} implies that both $(R^{2n-1}\pi_*
\bQ_{\xc})_{\mid\cc}$ and $(R^{2n}\pi_* \bQ_{\xc})_{\mid\cc}$ are
local systems on any component $\cc$ of $\vX_{\delta}$ having the
expected dimension. We observed that this also follows from Thom's
first isotopy lemma.
\end{remark}

\medskip
Fix again a nodal hypersurface $\xc_H$ such that $H\in \cc$. As
explained in \cite[2.11]{DeCMHodgeConj}, the decomposition
(\ref{decThm}) produces a \textit{perverse filtration} on the groups
$H^l(\xc _H)=H^l(X\cap H)$. By \cite[Lemma 3.1]{DeCMHodgeConj},
$H^{2n}(\xc _H)=H^{2n}_{\leq 1}(\xc _H)$ and the graded piece
$H^{2n}_1(\xc _H)$ coincides with the  stalk at $H$ of the constant
system $(R^{2n}\pi_* \bQ_{\xc})_U$. Furthermore,  in the proof of
\cite[Lemma 3.1]{DeCMHodgeConj} it is shown that $\hc^1
(IC((R^{2n-1}\pi_* \bQ_{\xc })_{ \mid U})[-N]) _H$ is one of the
summands of the graded piece $H^{2n}_0(\xc _H)$. By (\ref{th2}) of
Theorem \ref{main}, we conclude that these are the only non trivial
summands of  the perverse filtration of $H^{2n}(\xc _H)=H^{2n}(X\cap
H)$. In a nutshell, we have proved that \textit{the perverse
filtration  is as simple as possible because it consists of only two
pieces that vary in local systems over any component of Severi's
variety}.

\begin{corollary}\label{cormain}
With notations as above, let $H\in \cc$. In the perverse filtration of $H^{2n}(\xc _H)$ we have $H^{2n}_{\leq -1}(\xc _H)=0$ and
$H^{2n}_{0}(\xc _H)\cong\hc^1 (IC((R^{2n-1}\pi_* \bQ_{\xc })_{ \mid U})[-N]) _H$.
In other words, the sequence
$$ 0\to \hc^1 (IC((R^{2n-1}\pi_* \bQ_{\xc })_{ \mid U})[-N]) _{\mid\cc}\to (R^{2n}\pi_* \bQ_{\xc})_{\mid\cc}\to \mathbb{Q}^{h}
\to 0,$$
represents the perverse filtration of the local system $(R^{2n}\pi_* \bQ_{\xc})_{\mid\cc}$.
\end{corollary}
By \cite[Lemma 3.1 (ii)]{DeCMHodgeConj}, any primitive Hodge class
belongs to $H^{2n}_{\leq 0}(\xc _H)$.  In view of previous
Corollary, we have $H^{2n}_{\leq 0}(\xc _H)=H^{2n}_{0}(\xc
_H)\cong\hc^1 (IC((R^{2n-1}\pi_* \bQ_{\xc })_{ \mid U})[-N]) _H$ .
This implies that, \textit{under the hypothesis that our hyperplane
cuts $X$ in a set of nodes imposing independent conditions, the
pull-back of a primitive Hodge cycle coincide with its local
Green-Griffiths invariant}. In particular, we get  a different proof
of the fact that \textit{the local Green-Griffiths invariant detects
the local triviality of an Hodge cycle} \cite[Proposition 3.8
(ii)]{DeCMHodgeConj}.
\medskip

\begin{remark}\label{defect}
Assume now that $X$ is a projective space. Then $H^{2n-2}(X)$ is
generated by the hyperplane class and $H^{2n-2}(X\cap H)\not \cong
H^{2n-2}(X) $ iff the hypersurface $X\cap H$ has ``defect''
\cite[\textsection 6.4]{Dimca}. Specifically, $h^1
(IC((R^{2n-1}\pi_* \bQ_{\xc })_{ \mid U})[-N]) _H$ coincides with
the defect of $X\cap H$, which is constant on some connected
component of $\xc _{\delta}$. We note in passing that a  great deal
of work has been devoted to the study of the defect of  projective
nodal complete intersections, in relation to the number and the
position of the nodes. For instance, in the paper \cite{CCM} one can
find a construction of factorial complete intersections of
codimension 2, having (asymptotically) a maximal number of nodes
(see also \cite{JAG} and \cite{EFGII} for other results concerning
smooth projective varieties of codimension 2).
\end{remark}

\bigskip

\section{Examples}

\medskip
In this  section  we provide several examples of even-dimensional
smooth projective varieties equipped with linear systems containing
nodal hypersurfaces $D$, whose nodes impose independent conditions.

\subsection{Curves in a projective surface}
Let $X$ denotes a smooth projective surface embedded in some
projective space via a very ample divisor $H$. Let $C\subset X$ be a
smooth curve in $X$ and let $K_X$ denotes the canonical divisor of
$X$. If $n\gg 0$, we can assume
\begin{equation}
\label{hypC} H^1(\ic _{C,X}(n))=0, \,\,\, \text{and} \,\, \mid H-C \mid \,\, \text{very ample on X}.
\end{equation}
Under these conditions, the general curve $R\in \mid H-C \mid$ is
smooth, $\Delta:= R\cap C$ is reduced and $R\cup C\sim nH$ has only
ordinary double points.

\medskip
\begin{proposition}
\label{thmsurface} With notations as above, assume conditions
(\ref{hypC}) verified. Assume additionally that $C\cdot K_X<0$. Then
the set of nodes $\Delta=R\cap C$ imposes independent conditions to
the linear system  $\mid nH \mid$, that is to say $H^1(\ic_{\Delta,
X}(n))=0$.
\end{proposition}
\begin{proof}
From the cohomology exact sequence deduced from the following short
exact sequence
$$0\rightarrow \ic_{C,X}(n) \rightarrow \ic_{\Delta,X}(n)\rightarrow \ic_{\Delta, C}(n)\rightarrow 0,$$
and taking into account of $H^1(\ic _{C,X}(n))=0$ (remember
(\ref{hypC})), we see that it suffices to prove $H^1(\ic_{\Delta,
C}(n))=0$. On the other hand, by adjunction we have
$$\omega_C\cong \oc_C(K_X+C)\cong \oc_C(K_X+nH-R)\cong \ic_{\Delta, C}(K_X+nH).$$
In view of the hypothesis $C\cdot K_X<0$, we conclude at once:
$$H^1(\ic _{\Delta,C}(n))\cong H^1(\omega_C(-K_X))\cong H^0(\oc_C(K_X))=0.$$
\end{proof}

\begin{remark}
We observe that the hypothesis of Proposition above is satisfied
when either $K_X$ is negative or when $C$ is an exceptional curve.

\end{remark}

\medskip
\subsection{Defective hypersurfaces of $\mathbb{P} ^{2n}$}

Let $X=\mathbb{P} ^{2n}$, $L=\mathbb{P} ^n$ a linear subspace of dimension $n\geq 2$, and
$$
f=a_0x_0+\dots+a_{n-1}x_{n-1}=0
$$
a general hypersurface of degree $k\geq 2$, containing $L$. The singular locus $\Delta$ of $f=0$ is a set of $\delta=(k-1)^n$
nodes, complete intersection of type $(k-1,\dots,k-1)$ in $L$. Let
$\mathcal I_{\Delta}$ the ideal sheaf of $\Delta$ in $L$. Then, the Koszul complex of $\Delta$ in $L$ is:
$$
0\to \wedge^n E^*\to \wedge^{n-1} E^*\to\dots E^*\to \mathcal
I_{\Delta}\to 0,
$$
with $E={\mathcal O_{\mathbb P^n}(k-1)}^{n}$. We observe that
$\wedge^n E^*=\mathcal O_{\mathbb P ^n}(n(1-k))$.

\bigskip
The set $\Delta$ imposes independent conditions to the linear system
$|\mathcal O_X(k)|$ if and only if  $h^1(\mathbb P^n,\mathcal
I_{\Delta}(k))=0$. For this to happen it suffices
$$
h^n(\mathbb P ^n,\mathcal O_{\mathbb P ^n}(n(1-k)+k))=0,
$$
namely that
$$
k<\frac{2n+1}{n-1}.
$$
This is true when $n=2$ e $k\leq 4$, $n=3$
and $k\leq 3$, $n\geq 4$ e $k= 2$.

Since a hypersurface containing a $\mathbb{P} ^n$ has defect, in all
these cases the nodes impose independent conditions and $\hc^{1}
IC(R^{2n-1}\pi{_*}\mathbb{Q})_D[-N]$ is non-trivial.

\medskip

\subsection{Defective hypersurfaces in a complete intersection of quadrics}

Let $X=X^{2n}\subset \mathbb P^{2n+h}$ be a general complete
intersection of $h$ quadrics containing a linear subspace $L=\mathbb
P ^n$:
$$
X=\begin{cases} q_1=a_{0}^1x_0+\dots+a_{n+h-1}^1x_{n+h-1}=0\\
\dots\\
q_h=a_{0}^hx_0+\dots+a_{n+h-1}^hx_{n+h-1}=0.
\end{cases}
$$
We cut $X$ with a further general quadric $Q$, containing $L$, with equation:
$$
q=\alpha_{0}x_0+\dots+\alpha_{n+h-1}x_{n+h-1}=0.
$$
The singular locus $\Delta$ of $X\cap Q$ is the degeneracy locus  $D_h(\phi)$ of a general morphism
$$
\phi:\mathcal O_{\mathbb P ^n}^{h+1}\to \mathcal O_{\mathbb P ^n}(1)^{h+n}.
$$
Thus, $\Delta$ is a set of $\delta=\binom{n+h}{n}$ nodes. The Eagon-Northcott complex of $\phi$ is:
$$
0\to S^{n-1} \mathcal O_{\mathbb{P}^n}^{h+1}\to S^{n-2} \mathcal
O_{\mathbb{P}^n}^{h+1}\otimes \mathcal O_{\mathbb P ^n}(1)^{h+n}\to
S^{n-3} \mathcal O_{\mathbb P ^n}^{h+1}\otimes\wedge^2 \mathcal
O_{\mathbb P ^n}(1)^{h+n}\to\dots \quad\quad\quad\quad\quad\quad
$$
$$
\quad\quad\quad\quad\quad\quad\quad\quad\quad\quad\quad\quad\quad\quad\dots
\to \wedge^{n-1} \mathcal O_{\mathbb P ^n}(1)^{h+n}\to \mathcal
I_{\Delta}(h+n)\to 0,
$$
where $\mathcal I_{\Delta}$ is the ideal sheaf of $\Delta$ in
$L$.

\bigskip
The set $\Delta$ imposes independent conditions to the linear system  $|\mathcal
O_X(2)|$ if and only if $h^1(\mathbb P ^n,\mathcal
I_{\Delta}(2))=0$. By the Eagon-Northcott complex above, since
$S^{n-1} \mathcal O_{\mathbb P ^n}^{h+1}$ is a direct sum of $\mathcal
O_{\mathbb P ^n}$, we have that $h^1(\mathbb P ^n,\mathcal I_{\Delta}(2))=0$ as soon as
$$
h^n(\mathbb P ^n,\mathcal O_{\mathbb P ^n}(-h-n+2))=0,
$$
namely if
$$
h^0(\mathbb P ^n,\mathcal O_{\mathbb P ^n}(h-3))=0.
$$
The last condition is satisfied if $1\leq h\leq 2$. When $h=1$
we have $n+1$ nodes in $\mathbb P ^n$, if $h=2$ we have
$\frac{1}{2}(n+2)(n+1)$ nodes in $\mathbb P^n$.

Also in this case, since a hypersurface containing a $\mathbb{P} ^n$
has defect, the nodes impose independent conditions and $\hc^{1}
IC(R^{2n-1}\pi{_*}\mathbb{Q})_D[-N]$ is non-trivial.

\end{document}